\documentclass[11pt]{amsart}

\pdfoutput=1

\usepackage[text={420pt,660pt},centering]{geometry}

\usepackage{cancel}

\usepackage{graphicx}
\usepackage{MnSymbol}
\usepackage{mathtools} 
\usepackage[colorlinks=true, pdfstartview=FitV, linkcolor=blue, citecolor=blue, urlcolor=blue,pagebackref=false]{hyperref}
\usepackage{microtype}

\usepackage{bm}
\usepackage{dsfont}
\usepackage{mathrsfs}
\usepackage{xcolor}

\usepackage[normalem]{ulem}

\usepackage{tikz}

\newtheorem{proposition}{Proposition}
\newtheorem{theorem}[proposition]{Theorem}
\newtheorem{lemma}[proposition]{Lemma}
\newtheorem{corollary}[proposition]{Corollary}

\theoremstyle{remark}

\theoremstyle{definition}
\newtheorem{definition}[proposition]{Definition}

\numberwithin{equation}{section}
\numberwithin{proposition}{section}
\numberwithin{figure}{section}
\numberwithin{table}{section}

\newcommand{\N}{\mathbb{N}}

\newcommand{\R}{\mathbb{R}}

\newcommand{\E}{\mathbb{E}}

\newcommand{\C}{\mathcal{C}}

\renewcommand{\leq}{\leqslant}
\renewcommand{\geq}{\geqslant}

\renewcommand{\subset}{\subseteq}
\renewcommand{\bar}{\overline}
\renewcommand{\tilde}{\widetilde}

\renewcommand{\d}{\mathrm{d}}

\DeclareMathOperator{\supp}{supp}

\newcommand{\la}{\left\langle}
\newcommand{\ra}{\right\rangle}

\renewcommand{\H}{\mathsf{H}}

\newcommand{\pj}{\mathrm{p}}
\newcommand{\lf}{\mathrm{l}}
\newcommand{\cH}{\mathcal{H}}
\newcommand{\HJ}{\mathrm{HJ}}
\renewcommand{\j}{{(j)}}

\newcommand{\itr}{\mathsf{int}}

\begin{document}

\author[H.-B. Chen]{Hong-Bin Chen}
\address[H.-B. Chen]{Institut des Hautes \'Etudes Scientifiques, France}
\email{hbchen@ihes.fr}

\title{A PDE perspective on the Aizenman--Sims--Starr scheme}

\begin{abstract}
In \cite{mourrat2022parisi,mourrat2020extending}, the limit free energy of an enriched mixed $p$-spin spin glass model is identified with the solution of a Hamilton--Jacobi equation. In this note, we remark that the Aizenman--Sims--Starr scheme yields a subsolution to the equation.
\end{abstract}

\maketitle

\section{Introduction}

\subsection{Mixed \texorpdfstring{$p$}{p}-spin model and enrichment}
Let $(\beta_p)_{p\geq 1}$ be a sequence of nonnegative number such that $\xi(r) = \sum_{p\geq 1}\beta_p^2 r^p$ is finite for all $r\in\R$. For each $N\in\N$, let $\Sigma_N = \{-1,+1\}^N$, and $(H_N(\sigma))_{\sigma\in\Sigma_N}$ be a centered Gaussian field with covariance
\begin{align}\label{e.model}
    \E H_N(\sigma^1)H_N(\sigma^2) = N\xi\left(\frac{\sigma^1\cdot\sigma^2}{N}\right),\quad\forall \sigma^1,\sigma^2\in\Sigma_N.
\end{align}
On $\Sigma_N$, we put measure $P_N = P_1^{\otimes N}$ where $P_1$ is uniform on $\{-1,+1\}$. Define the free energy associated with $H_N(\sigma)$ by
\begin{align*}
    \mathscr{F}_N = \frac{1}{N}\E\log\int \exp H_N(\sigma)\d P_N(\sigma).
\end{align*}

Then, we enrich the model by adding an external field. Let us describe the set of parameters for the field.
Let $\mathcal{M}$ be the collection of nondecreasing right-continuous maps from $[0,1)$ to $[0,\infty)$, and $\mathcal{M}_\mathrm{b}$ be the subcollection of bounded maps in $\mathcal{M}$. Note that the right-continuous inverse of $\mu\in\mathcal{M}$ is the distribution function of some probability measure on $[0,\infty)$. For each $\mu\in\mathcal{M}_\mathrm{b}$, we set $\mu(1) = \lim_{s\to1}\mu(s)$.

Let $\mathfrak{R}$ be the Ruelle probability cascade on some Hilbert space such that the distribution of the inner product $\alpha^1\cdot \alpha^2$ under $\E \mathfrak{R}^{\otimes 2}$ is uniform on $[0,1]$ where $\alpha^1, \alpha^2$ are independent samples from $\mathfrak{R}$. 
Conditionally on $\mathfrak{R}$, let $(w^\mu(\alpha))_{\alpha\in\supp(\mathfrak{R})}$ be a centered Gaussian process with covariance
\begin{align*}
    \E w^\mu(\alpha^1)w^\mu(\alpha^2) = 2\mu(\alpha^1\cdot\alpha^2),\quad\forall \alpha^1,\alpha^2\in\supp(\mathfrak{R}).
\end{align*}

For $(t,\mu) \in [0,\infty)\times \mathcal{M}_\mathrm{b}$, we define (notice an additional minus sign)
\begin{align}\label{e.F_n}
    F_N(t,\mu) = - \frac{1}{N}\log \iint \exp{H_N^{t,\mu}(\sigma,\alpha)}\d P_N(\sigma) \d \mathfrak{R}(\alpha)
\end{align}
with enriched Hamiltonian
\begin{align}\label{e.enriched}
    H_N^{t,\mu}(\sigma,\alpha) = \sqrt{2t}H_N(\sigma) + \sum_{i=1}^N\sigma_i w_i^\mu(\alpha) 
    - N t\xi\left(N^{-1}|\sigma|^2\right)-\mu(1)|\sigma|^2
\end{align}
where $(w^\mu_i(\alpha))_{i\geq 1}$ are i.i.d.\ copies of $w^\mu(\alpha)$ conditioned on $\mathfrak{R}$ which are also independent of $H_N(\sigma)$. The last two terms in $H_N^{t,\mu}(\sigma,\alpha)$ are added to normalize $\E \exp H_N^{t,\mu}(\sigma,\alpha) = 1$. Due to $\sigma\in \Sigma_N$, we simply have $|\sigma|^2 =N$ in the above. 

We write $\bar F_N = \E F_N$.
By \cite[Proposition~3.1]{mourrat2020free} (see also \cite[Theorem~1]{guerra2001sum}),
\begin{align}\label{e.F_N-lip}
    \left|\bar F_N(t,\mu) - \bar F_N(t,\mu')\right|\leq \left|\mu -\mu'\right|_{L^1}
\end{align}
for all $\mu,\mu'\in\mathcal{M}_\mathrm{b}$ and $t\geq 0$.
Hence, we can extend $\bar F_N$ to $\mathcal{M}\cap L^1$ where $L^1 = L^1([0,1))$. Since it is more convenient to work in a Hilbert space, we set $\cH= L^2([0,1))$ and $\C = \mathcal{M}\cap\cH$. Note that $\C$ is a closed convex cone in $\cH$ and $\C$ has an empty interior with respect to $\cH$.

We can recover the original free energy via the relation
\begin{align*}\mathscr{F}_N=- \bar F_N\left(\frac{1}{2},0\right)+\frac{1}{2}\xi(1)
\end{align*}
where $0$ is the constantly zero map in $\mathcal{M}$.

Note that $\bar F_N(0,\cdot)$ does not depend on $N$. We thus set $\psi = \bar F_1(0,\cdot)$ which can be expressed as
\begin{align}\label{e.psi}
    \psi(\mu) =- \E\log \la \cosh(w^\mu(\alpha))\ra + \mu(1),\quad\forall \mu \in \mathcal{M}_\mathrm{b},
\end{align}
where $\la \cdot\ra =\mathfrak{R}$.

When $\mu$ is a step function with finite steps, the random array $(\mu(\alpha^l\cdot\alpha^{l'}))_{l,l'\in\N}$ forms a discrete Ruelle probability cascade, and $F_N$ can be expressed more familiarly in terms of a tree with weights given by Poisson--Dirichlet cascades (see \cite[(3.8)]{mourrat2020free}).

\subsection{Previous works and the main remark}

The Parisi formula (e.g.\ see \cite[Section~3.1]{pan}) gives the limit free energy $\lim_{N\to\infty}\mathscr{F}_N$.
For the Sherrington-Kirkpatrick Model (for $\xi(r) = \beta^2_2 r^2$), Parisi proposed the formula in \cite{parisi79,parisi80}, which is later mathematically verified in \cite{gue03,Tpaper,panchenko2014parisi} (see also \cite{Tbook1,Tbook2,pan}):
\begin{theorem}[\cite{gue03,Tpaper,panchenko2014parisi}]For the original model \eqref{e.model}, $\lim_{N\to\infty}\mathscr{F}_N=\inf_{\mu\in \mathcal{M}_{[0,1]}}\mathscr{P}(\mu)$.
\end{theorem}
\noindent Here, $\mathcal{M}_{[0,1]}$ is the collection of nondecreasing right-continuous maps from $[0,1)$ to $[0,1]$, and $\mathscr{P}(\mu)$ is the Parisi functional whose definition we refer to \cite[Section~3.1]{pan}.

As now clearly understood (see \cite{pan}), the upper bound of $\lim_{N\to\infty}\mathscr{F}_N$ follows from Guerra's RSB interpolation, and the lower bound can be obtained from the Aizenman--Sims--Starr scheme along with properties of the distributions of the overlap arrays that satisfy the Ghirlanda--Guerra identities. Let us briefly explain the latter. Setting $\mathscr{A}_N = (N+1)\mathscr{F}_{N+1} - N \mathscr{F}_{N}$, one gets $\mathscr{F}_N=\frac{1}{N}\sum_{j=0}^{N-1}\mathscr{A}_j$ and thus
\begin{align*}
    \liminf_{N\to\infty}\mathscr{F}_N \geq \liminf_{N\to\infty}\mathscr{A}_N.
\end{align*}
Then, from a subsequence realizing the infimum of $\mathscr{A}_N$, one extracts a further subsequence $(N_k)_{k\in\N}$ along which the Gibbs measure associated with $\mathscr{A}_{N_k}$ converges to an asymptotic Gibbs measure that satisfies the Ghirlanda--Guerra identities. Using properties of such measures, one proceeds to verify that
\begin{align*}
    \lim_{k\to\infty} \mathscr{A}_{N_k} = \mathscr{P}(\zeta)
\end{align*}
for some $\zeta\in\mathscr{M}_{[0,1]}$, which gives the lower bound $\liminf_{N\to\infty}\mathscr{F}_N \geq \inf_{\mu\in \mathcal{M}_{[0,1]}}\mathscr{P}(\mu)$.

Let us turn to a PDE perspective of the story.
Inspired by \cite{guerra2001sum, genovese2009mechanical, barra2010replica, barra2013mean} in the physics literature, J.-C.\ Mourrat mathematically initiated a PDE approach to recast the limit free energy of an enriched model as the solution to a Hamilton--Jacobi equation \cite{mourrat2018hamilton,mourrat2022parisi,mourrat2020extending,mourrat2020nonconvex,mourrat2020free}. 
It was proved in \cite{mourrat2022parisi,mourrat2020extending} that the limit free energy admits a representation by the Hopf--Lax formula:
\begin{theorem}[\cite{mourrat2022parisi,mourrat2020extending}]\label{t.hopf-lax}
For the enriched model~\eqref{e.enriched},
\begin{align}\label{e.hopf-lax_og}
    \lim_{N\to\infty}\bar F_N(t,\mu) =\sup_{\nu \in \C}\left(\psi(\nu) - t \int_0^1\xi^*\left(\frac{\nu(s)-\mu(s)}{t}\right)\d s\right),\quad\forall (t,\mu)\in[0,\infty)\times \C.
\end{align}
\end{theorem}
\noindent Here, $\psi$ was given in \eqref{e.psi}, and 
\begin{align}\label{e.xi^*}
    \xi^*(s) = \sup_{r\geq 0}\{r s - \xi(r)\},\quad\forall s\in\R.
\end{align}
The right-hand side of \eqref{e.hopf-lax_og} is expected to solve an infinite-dimensional Hamilton--Jacobi equation, formally expressed as
\begin{align}\label{e.hje_raw}
    \partial_t f - \int_0^1 \xi(\nabla f)\d s =0\quad\text{on $[0,\infty)\times\C$}.
\end{align}
Here, $\nabla$ is taking derivative in terms of $\mu\in\C\subset\cH$. Hence, for each $(t,\mu)$, $\nabla f(t,\mu)$ is understood to be an element in $\cH$. Writing $s\mapsto (\nabla f(t,\mu))(s)$, we interpret the integral in \eqref{e.hje_raw} at $(t,\mu)$ as $\int_0^1\xi((\nabla f(t,\mu))(s))\d s$.

The well-posedness of~\eqref{e.hje_raw} in the sense of viscosity solutions is later proved in~\cite{chen2022hamilton}. Due to technical issues in infinite dimensions, equation~\eqref{e.hje_raw} is modified into 
\begin{align}\label{e.hj}
    \partial_t f - \H(\nabla f) =0,\quad\text{on $[0,\infty)\times\C$}
\end{align}
where the nonlinearity $\H:\cH\to\R$ is a Lipschitz regularization of $\iota\mapsto \int_0^1\xi(\iota(s))\d s$, to be described in Section~\ref{s.nonl_terms}.

We want to remark that the Aizenman--Sims--Starr scheme yields something meaningful under this PDE perspective. For each $N\in\N$, set
\begin{align}\label{e.A_N}
    A_N = (N+1)\bar F_{N+1} - N \bar F_N.
\end{align}
\begin{proposition}\label{p.a-s-s}
For every $(t_0,\mu_0)$ in a dense subset of  $[0,\infty)\times \C$, every subsequence of $(A_N(t_0,\mu_0))_{N\in\N}$ has a further subsequence $(A_{N_k}(t_0,\mu_0))_{k\in\N}$ such that
\begin{align*}
    \lim_{k\to\infty} A_{N_k}(t_0,\mu_0) = \psi(t_0\nu+\mu_0)-t_0\int_0^1\xi^*(\nu(s))\d s
\end{align*}
for some $\nu\in\mathcal{M}_\mathrm{b}$. Moreover, the function
\begin{align*}
    [0,\infty)\times\C\ni(t,\mu)\mapsto \psi(t\nu + \mu) -t\int_0^1 \xi^*(\nu(s))\d s
\end{align*}
is a viscosity subsolution of \eqref{e.hj}.
\end{proposition}

In other words, any limit from the Aizenman--Sims--Starr scheme extends to a subsolution. Note that the above function is equal to $\psi$ at $t=0$. Also, the comparison principle of viscosity solutions implies that the solution of \eqref{e.hj} is the supremum of subsolutions. Hence, we can obtain a one-sided bound. (Recall that there is an additional minus sign in the definition of $F_N$ in \eqref{e.F_n}; hence, the following upper bound corresponds to a lower bound in the usual spin glass notation, as expected for the Aizenman--Sims--Starr scheme.)
\begin{corollary}\label{c}
Let $f$ be the unique viscosity solution to \eqref{e.hj} with initial condition $f(0,\cdot)=\psi$.
Then, $\limsup_{N\to\infty}\bar F_N\leq f$ pointwise everywhere on $[0,\infty)\times\C$.
\end{corollary}

\noindent
It is proven in \cite{mourrat2020nonconvex,mourrat2020free} that, for a wide class of spin glass models including models with nonconvex spin interactions, the limit free energy (with minus sign) is bounded from below by the viscosity solution, which gives $\liminf_{N\to\infty}\bar F_N\geq f$ here. This along with the above corollary implies that $\liminf_{N\to\infty}\bar F_N = f$ and thus the Hopf--Lax formula in Theorem~\ref{t.hopf-lax} indeed represents the viscosity solution.

\subsection*{Acknowledgements}
The author would like to thank Jean-Christophe Mourrat for helpful comments. This project has received funding from the European Research
Council (ERC) under the European Union’s Horizon 2020 research and innovation programme (grant agreement No.\ 757296).

\section{Proofs}

We first introduce definitions to make sense of equation~\eqref{e.hj} in Section~\ref{s.defs}. Then, we prove Proposition~\ref{p.a-s-s} and its corollary after stating two lemmas in Section~\ref{s.pf_main}. Finally, we prove these two lemmas in Sections~\ref{s.pf_ASS} and \ref{s.pf_subsol}.

\subsection{Definitions}\label{s.defs}
We first define differentiability in infinite dimensions, then construct the modified nonlinearity $\H$ appearing in equation~\eqref{e.hj}, and lastly define the notion of viscosity solutions.
\subsubsection{Derivatives}
We define derivatives in the Fr\'echet sense. 
For $(t,\mu)\in [0,\infty)\times\C$, we interpret $t$ as the temporal variable and $\mu$ the spatial variable. 
For any Hilbert space $\mathcal{X}$, we often denote its inner product and norm as $\la\cdot,\cdot\ra_\mathcal{X}$ and $|\cdot|_\mathcal{X}$, respectively. A function $\phi:(0,\infty)\times\C\to\R$ is said to be \textit{differentiable} at $(t,\mu)$ if there is $\mathbf{v} \in \R\times\cH$ such that
\begin{align*}
    \phi(s,\nu)=\phi(t,\mu) + \la \mathbf{v},(s-t,\nu-\mu)\ra_{\R\times\cH} +o(|(s-t,\nu-\mu)|_{\R\times\cH})
\end{align*}
as $(s,\nu)\to(t,\mu)$ in $ (0,\infty)\times\C$. Note that $\C$ is not open and has an empty interior in $\cH$. However, since the closure of the span of $\C$ is $\cH$, we can verify that, if such $\mathbf{v}$ exists, it must be unique and we denote it by $(\partial_t\phi(t,\mu), \nabla \phi(t,\mu))$. An everywhere differentiable function $\phi:(0,\infty)\times\C$ is said to be \textit{smooth} if, at every $(t,\mu)$,
\begin{align*}
    \phi(s,\nu)=\phi(t,\mu) + \la (\partial_t\phi(t,\mu), \nabla \phi(t,\mu)),\,(s-t,\nu-\mu)\ra_{\R\times\cH} +O\left(|(s-t,\nu-\mu)|^2_{\R\times\cH}\right)
\end{align*}
as $(s,\nu)\to(t,\mu)$ in $(0,\infty)\times\C$.

\subsubsection{Nonlinear terms}\label{s.nonl_terms}
For any function $\varphi:\R\to\R $, we set
\begin{align}\label{e.def_H_phi}
    \H_\varphi(\iota) = \int_0^1\varphi(\iota(s))\d s
\end{align}
for all $\iota \in\cH$ such that the right-hand side is finite. Then, the formal equation~\eqref{e.hje_raw} can be rewritten as $\partial_t f - \H_\xi(\nabla f)=0$.

Previously, we mentioned that the nonlinearity $\H$ in \eqref{e.hj} is a modification of $\H_\xi$. Here, we define $\H$ precisely. First, let $\bar \xi:\R\to \R$ be a regularization of $\xi$ such that $\bar\xi =\xi $ on $[0,1]$, and $\bar\xi$ is Lipschitz, nondecreasing, bounded from below, and convex. 
It is easy to see that $\bar \xi$ exists (see \cite[Lemma~4.2]{chen2022hamilton} for proof). We define
\begin{align*}
    \H(\iota) = \inf\left\{\H_{\bar\xi}(\mu):\mu \in \C\cap (\iota+\C^*)\right\},\quad\forall \iota\in\cH,
\end{align*}
where $\C^*$ is the dual cone of $\C$ in $\cH$ defined by $\C^* = \{\iota\in\cH:\la\iota,\mu\ra_\cH\geq 0,\ \forall \mu\in\C\}$.

Then, we can check that $\H:\cH\to\R$ is Lipschitz, bounded from below, convex, and $\C^*$-nondecreasing in the sense that
\begin{align}\label{e.C^*-nond}
    \H(\iota)\geq \H(\iota')\quad\text{if }\iota-\iota'\in\C^*
\end{align}
(see \cite[Lemma~4.4]{chen2022hamilton} for proof). Since $\bar\xi=\xi$ on $[0,1]$, we have
\begin{align}\label{e.H=H_xi}
    \H(\mu ) =\H_\xi(\mu) \quad\text{if $\mu\in\C$ and $0\leq \mu\leq 1$ a.s.}
\end{align}

We need this modification for technical reasons due to the infinite dimensionality of the equation. (One can show that the viscosity solution of equation~\eqref{e.hj} does not depend on modifications, but this is not important for the purpose here.) 

\subsubsection{Viscosity solutions}
We make sense of \eqref{e.hj} in the viscosity sense.
\begin{definition}[Viscosity solutions]\label{d.vs}
A continuous function $f:[0,\infty)\times\C\to\R$ is a \textit{viscosity subsolution} (resp.\ \textit{supersolution}) of \eqref{e.hj} if, whenever $f-\phi$ achieves a local maximum (resp.\ minimum) at some $(t,\mu)\in (0,\infty)\times\C$ for some smooth function $\phi:(0,\infty)\times\C\to\R$, it holds that
\begin{align*}
    \left(\partial_t\phi - \H(\nabla\phi)\right)(t,\mu)\leq 0 \quad\text{(resp.\ $\geq 0$)}.
\end{align*}
If $f$ is both a viscosity subsolution and supersolution of \eqref{e.hj}, we call $f$ a \textit{viscosity solution} of \eqref{e.hj}.
\end{definition}

\noindent Here, a local extremum is understood to be an extremum over some closed bounded ball of $\R\times\cH$ intersected with $(0,\infty)\times\C$.

The well-posedness of equation \eqref{e.hj} in the viscosity sense, consisting of the comparison principle and the existence of Lipschitz solutions, was verified in \cite[Section~3]{chen2022hamilton}. In particular, \eqref{e.hj} has a unique viscosity solution $f$ with initial condition $f(0,\cdot)=\psi$ and $f$ is Lipschitz (with the respect to $|\cdot|_{\R\times\cH}$ on $\R_+\times\C$).

\subsection{Proofs of Proposition~\ref{p.a-s-s} and Corollary~\ref{c}}\label{s.pf_main}

Let $\mathcal{M}_\mathrm{atom}$ be a subset of $\C$ consisting of nondecreasing step functions  $\sum_{k=1}^n a_k \mathds{1}_{[s_{k-1},s_k)}$ where $a_k\leq a_{k+1}$, $0=s_0<s_1<\cdots<s_n=1$, and $n\in\N$. Throughout, we let $U$ be the uniform random variable on $[0,1)$. Note that, for every $\mu\in \mathcal{M}_\mathrm{atom}$, the law of $\mu(U)$ is an atomic measure. Recall the definition of $A_N$ in~\eqref{e.A_N}.
\begin{lemma}[Aizenman--Sims--Starr Scheme]\label{l.ASS}
For each $(t_0,\mu_0) \in [0,\infty)\times \mathcal{M}_\mathrm{atom}$, every subsequence of $(A_N(t_0,\mu_0))_{N\in\N}$ has a further subsequence $(A_{N_k}(t_0,\mu_0))_{k\in\N}$ such that
\begin{align*}
    \lim_{k\to\infty} A_{N_k}(t_0,\mu_0) = \psi(t_0\nu+\mu_0)-t_0\int_0^1\xi^*(\nu(s))\d s
\end{align*}
for some $\nu\in\mathcal{M}_\mathrm{b}$.
\end{lemma}
\begin{lemma}\label{l.subsol}
Let $\nu \in \C$ and define
\begin{align}\label{e.g}
    g(t,\mu) = \psi(t\nu + \mu) -t\int_0^1 \xi^*(\nu(s))\d s,\quad\forall (t,\mu)\in [0,\infty)\times\C.
\end{align}
Then, $g$ is a viscosity subsolution of \eqref{e.hj}.
\end{lemma}

\begin{proof}[Proof of Proposition~\ref{p.a-s-s}]
Combining the two lemmas above, we obtain the desired result with the dense set $[0,\infty)\times \mathcal{M}_\mathrm{atom}$.
\end{proof}

To prove the corollary, we need the comparison principle of viscosity solutions~\cite[Proposition~3.8]{chen2022hamilton}:

\begin{proposition}[Comparison principle \cite{chen2022hamilton}]\label{p.cp}
If $u$ and $v$ are Lipschitz, and they are respectively a viscosity subsolution and a viscosity supersolution of \eqref{e.hj}, then $\sup_{[0,\infty)\times\C}(u-v)= \sup_{\{0\}\times \C}(u-v)$.
\end{proposition}

\begin{proof}[Proof of Corollary~\ref{c}]
Proposition~\ref{p.cp} implies that
\begin{align*}f(t,\mu) = 
    \sup\left\{g(t,\mu):g\in\mathcal{S} \right\},\quad\forall (t,\mu)\in[0,\infty)\times\C,
\end{align*}
where $\mathcal{S} =\left\{g:\text{ $g$ is a Lipschitz viscosity subsolution of \eqref{e.hj} satisfying $g(0,\cdot)\leq \psi$} \right\}$.
On the other hand, by \eqref{e.F_N-lip} and H\"older's inequality, we have that $\bar F_N$ is Lipschitz in $\mu$ uniformly in $t$ and $N$. By computing the derivative, one can easily see that $\bar F_N$ is also Lipschitz in $t$ uniformly in $\mu$ and $N$.
Hence, to show $\limsup_{N\to\infty}\bar F_N\leq f$, it suffices to show that
\begin{align}\label{e.recipe}
    \limsup_{N\to\infty}\bar F_N(t_0,\mu_0) \leq \sup\{g(t_0,\mu_0):g\in\mathcal{S}\}
\end{align}
for $(t_0,\mu_0)$ in a dense subset of $[0,\infty)\times\C$. Let $A_N$ be given in \eqref{e.A_N}, and we have that $\limsup_{N\to\infty} \bar F_N(t_0,\mu_0)\leq \limsup_{N\to\infty} A_N(t_0,\mu_0)$. Taking a subsequence realizing $\limsup_{N\to\infty} A_N(t_0,\mu_0)$ and extracting a further subsequence given in Proposition~\ref{p.a-s-s}, we deduce that $\limsup_{N\to\infty} A_N(t_0,\mu_0)= g_0(t_0,\mu_0)$ for some $g_0\in \mathcal{S}$ that is clearly Lipschitz and satisfies $g_0(0,\cdot)=\psi$. This implies \eqref{e.recipe} and completes the proof.
\end{proof}

In the remainder of this section, we prove Lemmas~\ref{l.ASS} and \ref{l.subsol}.

\subsection{Proof of Lemma~\ref{l.ASS}}\label{s.pf_ASS}

This lemma follows from a combination of known results. The first step is to compute the term $A_N$ in the Aizenman--Sims--Starr scheme, following the argument in \cite[Section~4.2]{mourrat2020extending}. 
The important tool is the synchronization mechanism developed in \cite{pan.multi,pan.potts,pan.vec} based on the overlap ultrametricity proved in \cite{pan.aom}. 
The second step is to identify the formula for the limit of $A_N$ with the desired expression, following the proof of \cite[Proposition 4.1]{mourrat2022parisi}.
For the first step, since the techniques and computations involved are standard, we give a sketch of the argument and refer to existing works for detail.

\begin{proof}[Proof of Lemma~\ref{l.ASS}]
We follow the announced two steps above.

\textbf{Step~1.}
As usual, we can slightly perturb the Gibbs measure associated with $\bar F_N(t_0,\mu_0)$ to ensure the synchronization between the spin overlap and the overlap of the Ruelle probability cascade. Define
\begin{align*}
    \mathscr{A}_N = -\big((N+1)\bar F_{N+1}(t_0,\mu_0)- N\bar  F_N(t_0,\mu_0)\big) + t_0\xi(1)+\mu_0(1),
\end{align*}
and we have
\begin{align}\label{e.F_N<A_N}
    A_N(t_0,\mu_0) =  -\mathscr{A}_N + t_0\xi(1)+\mu_0(1).
\end{align}

Let $\xi'$ be the derivative of $\xi$ and define $\theta(r) = r\xi'(r)-\xi(r)$ for all $r\in\R$. Following the computation in the proof of \cite[Theorem~3.6]{pan}, we can get
\begin{align}\label{e.sA_N}
    \mathscr{A}_N= \E\log \la \cosh\left(z_N(\sigma,\alpha)\right)\ra - \E\log \la \exp y_N(\sigma)\ra +o(1)
\end{align}
where $\la \cdot\ra$ is the Gibbs measure associated with Hamiltonian $H'_N(\sigma) + \sum_{i=1}^N\sigma_i w_i^\mu(\alpha)$ plus a perturbation. Here, $(H'_N(\sigma))_{\sigma\in\Sigma_N}$ is a centered Gaussian process satisfying
\begin{align*}
    \E H'_N(\sigma^1)H'_N(\sigma^2) = 2t_0(N+1)\xi\left(\frac{\sigma^1\cdot\sigma^2}{N+1}\right)\end{align*}
and $(z_N(\sigma,\alpha))_{\sigma,\alpha}$ and $(y_N(\sigma))_\sigma$ are two centered Gaussian processes, independent from each other and $H'_N(\sigma)$, with covariances
\begin{gather*}
    \E z_N(\sigma^1,\alpha^1)z_N(\sigma^2,\alpha^2) = 2t_0\xi'\left(\frac{\sigma^1\cdot\sigma^2}{N}\right) + 2\mu_0(\alpha^1\cdot\alpha^2),
    \\
    \E y_N(\sigma^1)y_N(\sigma^2) = 2t_0 \theta\left(\frac{\sigma^1\cdot\sigma^2}{N}\right).
\end{gather*}
We define
\begin{align*}
    R^N = \left(R^{\mathrm{s},N}_{l,l'},\  R^{\mathrm{r},N}_{l,l'}\right)_{l,l'\in\N} =\left(\frac{\sigma^l\cdot\sigma^{l'}}{N},\  \mu_0 (\alpha^l\cdot\alpha^{l'})\right)_{l,l'\in\N}.
\end{align*}
Given any subsequence, we can extract a further subsequence $(N_k)_{k\in\N}$ along which the law of $R^N$ converges to some random array $R=(R^\mathrm{s},R^\mathrm{r})$ (in terms of any finite-dimensional projection). By the invariance property of the Ruelle probability cascade (\cite[Theorem~4.4]{pan}), we still have $R^\mathrm{r}_{1,2} \stackrel{\mathrm{d}}{=}\mu_0(U)$. 
Suppose that $R^\mathrm{s}_{1,2}\stackrel{\mathrm{d}}{=}\zeta(U)$ for some $\zeta\in\mathcal{M}_\mathrm{b}$. Then, by the synchronization mechanism (\cite[Theorem~4]{pan.multi} together with \cite[Proposition~5.2]{mourrat2020free}), we can deduce that $(R^\mathrm{s}_{1,2},R^\mathrm{r}_{1,2})\stackrel{\mathrm{d}}{=}(\zeta(U),\mu_0(U))$. 

Due to $\frac{\sigma^1\cdot\sigma^{2}}{N}\in[0,1]$, we have that $\zeta([0,1))\subset [0,1]$.
Since discrete Ruelle probability cascades are easier to compute, we can approximate $\zeta$ by elements in $\mathcal{M}_\mathrm{atom}$, and then pass to the limit. Hence, let us assume $\zeta\in \mathcal{M}_\mathrm{atom}$. Moreover, we can always add a small mass near $1$ so that $\zeta(1) =1$ (recall that $\zeta(1) = \lim_{r\to1}\zeta(r)$ as a definition), which is needed for a technical reason in Step~2.

We redefine $\la \cdot\ra$ to be a random measure on a Hilbert space such that $R\stackrel{\mathrm{d}}{=}(\tau^l\cdot\tau^{l'})_{l,l'\in\N}$ (along any finite-dimensional projection) under $\E\la\cdot\ra^{\otimes\infty}$ for $(\tau^l)_{l\in\N}$ independently sampled from $\la \cdot\ra$. Setting $\tau=\tau^1$, we obtain from \eqref{e.sA_N} that
\begin{align}\label{e.limA_N}
    \lim_{k\to\infty} \mathscr{A}_{N_k} = \E\log \la \cosh\left(z(\tau)\right)\ra - \E\log \la \exp y(\tau)\ra
\end{align}
where the centered Gaussian processes satisfy
\begin{gather*}
    \E z(\tau^1)z(\tau^2) \stackrel{\mathrm{d}}{=} 2t_0\xi'\left(\zeta(U)\right) + 2\mu_0(U),\qquad\E y_N(\tau^1)y_N(\tau^2) \stackrel{\mathrm{d}}{=} 2t_0 \theta\left(\zeta(U)\right).
\end{gather*}
The expectations on the left-hand sides are conditioned on $(\tau^1,\tau^2)$.

\textbf{Step~2.}
We identify terms on the right-hand side of \eqref{e.limA_N}.
Using~\eqref{e.psi} with $t_0\xi'(\zeta)+\mu_0$ substituted for $\mu$, we have that
\begin{align*}
    \E\log \la \cosh\left(z(\tau)\right)\ra = -\psi(t_0\xi'(\zeta)+\mu_0)+(t_0\xi'(\zeta)+\mu_0)(1).
\end{align*}
The second term on the right of \eqref{e.limA_N} has the same form as that of the second term in the Parisi formula expressed in \cite[Lemma~3.1]{pan} which is equal to the second term on the right of \cite[(3.11)]{pan}. Since we have assumed $\zeta\in\mathcal{M}_\mathrm{atom}$ and $\zeta(1)=1$, we can write $\zeta=\sum_{l=1}^nq_l\mathds{1}_{[\zeta_{l-1},\zeta_l)}$ for $0=\zeta_0<\zeta_1<\dots<\zeta_{n}=1$ and $0\leq q_0\leq q_1\leq\cdots q_n=1$. Using the computation in \cite[Lemma~3.1]{pan} and summation by parts, we can get
\begin{align*}
    \E\log \la \exp y(\tau)\ra & = \frac{1}{2}\sum_{0\leq l\leq n-1}\zeta_l\left(\tilde\theta(q_{l+1})-\tilde\theta(q_l)\right) = \frac{1}{2}\left(-\sum_{l=1}^n (\zeta_{l}-\zeta_{l-1})\tilde\theta(q_l) + \zeta_n\tilde\theta(q_n)-\zeta_0\tilde\theta(q_0)\right)
    \\
    & =\frac{1}{2}\left( -\int_0^1\tilde\theta(\zeta(s))\d s + \tilde\theta(1)\right)
\end{align*}
where $\tilde \theta =2t_0\theta$.
Inserting these into \eqref{e.limA_N}, we arrive at
\begin{align*}
    \lim_{k\to\infty} \mathscr{A}_N = -\psi(t_0\xi'(\zeta)+\mu_0)+(t_0\xi'(\zeta)+\mu_0)(1) -t_0\theta(\zeta(1))+ t_0\int_0^1\theta(\zeta(s))\d s .
\end{align*}
Inserting this to \eqref{e.F_N<A_N}, we get 
\begin{align*}
    \lim_{k\to\infty}A_{N_k}(t_0,\mu_0) = \psi(t_0\xi'(\zeta)+\mu_0) - t_0\int_0^1\theta(\zeta(s))\d s +t_0\Big(\xi(1)-\xi'(\zeta(1))+\theta(\zeta(1))\Big).
\end{align*}
Due to the definition of $\theta$ and $\zeta(1)=1$, we have $\xi(1)-\xi'(\zeta(1))+\theta(\zeta(1))=0$. Recall the definition of $\xi^*$ in \eqref{e.xi^*}. Using the convexity of $\xi$ on $[0,\infty)$ and $\xi'(r)\geq 0$ for $r\geq 0$, we have that
\begin{align*}
    \theta(r) = r\xi'(r)-\xi(r) =\sup_{q\geq 0}\{q\xi'(r)-\xi(q)\}=\xi^*(\xi'(r)),\quad\forall r\geq 0.
\end{align*}
Therefore,
\begin{align*}
    \lim_{k\to\infty}A_{N_k}(t_0,\mu_0) = \psi(t_0\xi'(\zeta)+\mu_0) - t_0\int_0^1\xi^*(\xi'(\zeta(s)))\d s.
\end{align*}
Setting $\nu=\xi'(\zeta)$, we obtain the desired result.
\end{proof}

\subsection{Proof of Lemma~\ref{l.subsol}}\label{s.pf_subsol}
We first explain the heuristics, then introduce the necessary notation, and lastly prove the lemma.
\subsubsection{Heuristics}
In notation \eqref{e.def_H_phi}, we can rewrite \eqref{e.g} as $g(t,\mu) = \psi(t\nu+\mu) - t\H_{\xi^*}(\nu)$. Since $\H$ is a modification of $\H_\xi$, we can believe that its convex conjugate $\H^*$ is effectively equal to $\H_{\xi^*}$. Then, we can formally compute
\begin{align*}
    \left(\partial_t g - \H(\nabla g)\right)(t,\mu) = \la\nabla\psi(t\nu+\mu),\nu\ra_\cH - \H^*(\nu) - \H(\nabla\psi(t\nu+\mu))\leq 0
\end{align*}
by the property of convex conjugates. This formally verifies that $g$ is a subsolution. There are two issues with this computation: we do not know if $\psi$ is differentiable and we need to justify that $\H^*$ is effectively $\H_{\xi^*}$. We use finite-dimensional approximations to make it rigorous.

\subsubsection{Notation for approximations}
For each $j\in\N$, we associate a dyadic partition $([\frac{k-1}{2^j},\frac{k}{2^j}))_{1\leq k\leq 2^j}$ of $[0,1)$. Recall that $\cH = L^2([0,1))$. For each $j$, we define $\cH^j$ to be $\R^{2^j}$ equipped with inner product $\la a,b\ra_{\cH^j} = \sum_{k=1}^{2^j}\frac{1}{2^j} a_kb_k$. Then, we define projection $\pj_j:\cH\to\cH^j$ and lift $\lf_j :\cH^j\to\cH$ by
\begin{gather*}
    \pj_j \iota = \left(2^j\int_{\frac{k-1}{2^j}}^\frac{k}{2^j}\iota(s)\d s\right)_{1\leq k \leq 2^j},\quad\forall \iota \in\cH; \qquad\qquad
    \lf_j a = \sum_{k=1}^{2^j}a_k \mathds{1}_{[\frac{k-1}{2^j}, \frac{k}{2^j})},\quad\forall a \in \cH^j.
\end{gather*}
For $\iota\in\cH$, we define $\iota^\j = \lf_j\pj_j\iota$ as a locally-averaged approximation of $\iota$. Define $\C^j = \pj_j(\C)$.

For a function $g$ defined on a subset of $\R\times \cH$ or $\cH$, we define its $j$-projection $g^j$ by $g^j(t,a) = g(t,\lf_j a)$ or $g^j(a) = g(\lf_ja)$. For a function $h$ defined on a subset of $\R\times \cH^j$ or $\cH^j$, we define its lift $h^\uparrow$ by $h^\uparrow(t,\mu) = h(t,\pj_j\mu)$ or $h^\uparrow(\mu) = h(\pj_j\mu)$.

We denote equation \eqref{e.hj} by $\HJ(\cH,\C,\H)$ and its Cauchy problem with initial condition $\psi$ by $\HJ(\cH,\C,\H;\psi)$. On $\C^j\subset \cH^j$, we can define differentiability and viscosity solutions in the same way as in Section~\ref{s.defs}. We denote by $\HJ(\cH^j,\C^j,\H^j)$ the equation with $\cH,\C,\H$ in \eqref{e.hj} replaced by $\cH^j,\C^j,\H^j$, and make sense of its solutions in the viscosity sense as in Definition~\ref{d.vs}. This equation is viewed as a finite-dimensional approximation of equation~\eqref{e.hj}.

We see that $\H^j$, the $j$-projection of $\H$, is Lipschitz, bounded below, convex, and $(\C^j)^*$-nondecreasing in a sense analogous to \eqref{e.C^*-nond}, where
\begin{align*}(\C^j)^* = \{a\in \cH^j:\la a,b\ra_{\cH^j}\geq 0,\, \forall b\in\C^j\}
\end{align*}
is the dual cone of $\C^j$ in $\cH^j$.

The set of dyadic partitions is well-ordered by the set inclusion, and it is a \textit{directed set} in the sense that, for any two partitions, there is a partition that refines both. In \cite[Section~3]{chen2022hamilton}, results were stated for objects indexed by a general directed set of partitions, which are applicable here.

We use superscript $j$ to indicate that an object is obtained through a projection-related operation and lower script $j$ for a generic index.

\begin{proof}[Proof of Lemma~\ref{l.subsol}]
For each $j$, define
\begin{align}\label{e.g_j}
    g_j(t,q) = \psi^j(t\pj_j \nu + q) - t \sum_{k=1}^{2^j}\frac{1}{2^j}\xi^*\left((\pj_j\nu)_k\right),\quad\forall (t,q)\in [0,\infty)\times \C^j.
\end{align}
It was shown in \cite[Proposition~3.9]{chen2022hamilton} that the local uniform limit of $\HJ(\cH^j,\C^j,\H^j)$ is a subsolution of \eqref{e.hj}. Hence,
it suffices to show that $g_j$ is a viscosity subsolution of $\HJ(\cH^j,\C^j,\H^j)$ and $g^\uparrow_j$ converges to $g$ locally uniformly along a subsequence. We proceed in steps.

\textbf{Step~1.}
We verify the local uniform convergence of $g^\uparrow_j$. Assuming
\begin{align}\label{e.mu-mu^j}
    \left|\mu - \mu^\j\right|_{L^1}\leq 2^{\frac{3-j}{2}}|\mu|_\cH,\quad\forall \mu \in \C,
\end{align}
and using \eqref{e.F_N-lip} for $t=0$, we have
\begin{align*}
    \left|\big(\psi^j(t\pj_j\nu+\cdot)\big)^\uparrow(\mu) - \psi(t\nu+\mu)\right| = \left|\psi\left(t\nu^\j +\mu^\j \right) -\psi (t\nu+\mu)\right|
    \\
    \leq \left|t\nu^\j +\mu^\j - t\nu-\mu\right|_{L^1}\leq 2^\frac{3-j}{2}|t\nu +\mu|_\cH.
\end{align*}
This implies that $(\psi^j(t\pj_j\nu+\cdot))^\uparrow $ converges locally uniformly to $\psi(t\nu+\cdot)$.
We postpone the proof of \eqref{e.mu-mu^j}.

Viewing $\nu^\j$ as the expectation (with respect to the Lebesgue measure on $[0,1)$) of $\nu$ conditioned on the sigma-algebra generated by the dyadic partition associated with $j$, we see that $(\nu^\j)_{j\in\N}$ is a martingale. By the martingale convergence theorem, we have $\lim_{j\to\infty} \nu^\j =\nu$ in $\cH$ (see \cite[Lemma~3.3~(7)]{chen2022hamilton} for detail).
Hence, we can extract a subsequence $(j_n)$ along which $\nu^{(j_n)}$ converges to $\nu$ a.e. Since $\xi^*$ (defined in \eqref{e.xi^*}) is bounded below ($\xi^*\geq  -\xi(0)$) and is lower-semicontinuous, by Fatou's lemma, we have
\begin{align*}
    \int_0^1\xi^*(\nu(s))\d s \leq \liminf_{n\to\infty}\int_0^1\xi^*\left(\nu^{(j_n)}(s)\right)\d s.
\end{align*}
On the other hand, by Jensen's inequality and the convexity of $\xi^*$,
\begin{align*}
    \int_0^1\xi^*\left(\nu^\j(s)\right)\d s \leq \int_0^1\xi^*(\nu(s))\d s,\quad\forall j\in\N.
\end{align*}
Hence, we must have
\begin{align*}
    \lim_{n\to\infty}\int_0^1\xi^*\left(\nu^{(j_n)}(s)\right)\d s = \int_0^1\xi^*(\nu(s))\d s .
\end{align*}
Recognizing that $\int_0^1\xi^*\left(\nu^{(j_n)}(s)\right)\d s $ is the sum in the definition of $g_{j_n}$, we conclude together with the previous paragraph that $g_{j_n}^\uparrow$ converges to $g$ (defined in \eqref{e.g}) locally uniformly.

\textbf{Step~2.}
We show that $g_j$ is a viscosity subsolution of $\HJ(\cH^j,\C^j,\H^j)$. First, we show that $g_j$ satisfies 
\begin{align}\label{e.g_j_eqn}
    \partial_t g_j - \H^j_\xi(\nabla g_j)\leq 0
\end{align}
in the classical sense, where $\H^j_\xi$ is the $j$-projection of $\H_\xi$ and is equal to $\sum_{k=1}^{2^j}\frac{1}{2^j}\xi(a_k)$ for $a\in \cH^j$.

As preparation, we show
\begin{align}\label{e.grad_psi_in}
    \nabla\psi^j(q)\in\C^j\cap [0,1]^{2^j},\quad\forall q \in \C^j.
\end{align}
Recall that $\psi=\bar F_1(0,\cdot)$ and the differentiability of $\bar F_1(0,\lf_j\cdot)$ restricted to discrete paths in $\C^j$ is known. Its derivatives were computed, for instance, in \cite[(2.17)]{mourrat2022parisi}.
By \eqref{e.F_N-lip} for $t=0$, we have
\begin{align*}
    |\psi^j(q) - \psi^j(q')|=|\psi(\lf_jq)-\psi(\lf_jq')|\leq |\lf_jq - \lf_jq'|_{L^1} = \sum_{k=1}^{2^j}2^{-j} |q_k-q'_k| = |q=q'|_{l^1},
\end{align*}
which implies that $|\nabla\psi^j(q)|_{l^\infty}\leq 1$ and thus $(\nabla \psi^j(q))_k\in [-1,1]$ for all $q$ and $k$. Using \cite[Proposition~3.6]{mourrat2020free} (with $\frac{k}{2^j}$ substituted for $\zeta_k$ therein), we have, for all $(t,q)\in [0,\infty)\times \C^j$ and $N\in\N$,
\begin{align*}
    0\leq\left(\nabla (\bar F_N)^j(t,q)\right)_k \leq \left(\nabla (\bar F_N)^j(t,q)\right)_{k+1},\quad\forall k.
\end{align*}
Applying this with $t= 0$, together with the previous result, we obtain \eqref{e.grad_psi_in}.

Now back to showing \eqref{e.g_j_eqn}, fixing any $t,q$, setting $p = \pj_j\nu$, and writing $\nabla\psi^j = \nabla \psi^j(tp+q)$, we obtain from \eqref{e.g_j} that
\begin{align}
    \left(\partial_t g_j - \H^j_\xi(\nabla g_j)\right)(t,q) = \la \nabla \psi^j,p\ra_{\cH^j} -\sum_{k=1}^{2^j}\frac{1}{2^j}\xi^*(p_k) - \sum_{k=1}^{2^j}\frac{1}{2^j}\xi((\nabla\psi^j)_k)\notag\\
    = \sum_{k=1}^{2^j}\frac{1}{2^j}\Big((\nabla\psi^j)_kp_k - \xi^*(p_k)-\xi\left((\nabla\psi^j)_k\right)\Big)\leq 0,\label{e.g_satisfy<}
\end{align}
where the last inequality follows from the definition of $\xi^*$ in \eqref{e.xi^*} and $(\nabla\psi^j)_k\geq 0$ due to \eqref{e.grad_psi_in}.
In particular, $g$ satisfies \eqref{e.g_j_eqn} in the classical sense on $(0,\infty)\times \itr(\C^j)$ where $\itr(\C^j)$ is the interior of $\C^j$. 

Since \eqref{e.H=H_xi} implies that
\begin{align*}\H^j = \H^j_\xi \quad\text{on $\C^j\times[0,1]^{2^j}$},
\end{align*}
by \eqref{e.grad_psi_in}, we have $\H^j(\nabla \psi^j) = \H^j_\xi(\nabla\psi^j)$. Due to this and \eqref{e.g_satisfy<}, $g$ solves $\partial_t g_j -\H^j(\nabla g_j)\leq 0$ in the classical sense on $(0,\infty)\times \itr(\C^j)$. 

Since $\itr(\C^j)$ is open, $g$ is a viscosity subsolution of $\HJ(\cH^j,\itr(\C^j),\H^j)$ (defined as in Definition~\ref{d.vs} with $\cH,\C,\H$ replaced by $\cH^j,\itr(\C^j),\H^j$).
Since $\H^j$ is $(\C^j)^*$-nondecreasing, \cite[Proposition~2.1]{chen2022hamiltonCone} implies that $g$ is a viscosity subsolution of $\HJ(\cH^j,\C^j,\H^j)$ (monotonicity of $\cH^j$ makes the boundary irrelevant).

\textbf{Step~3.}
It remains to prove \eqref{e.mu-mu^j}. Setting $\Delta = \frac{1}{2^j}$ and using that $\mu$ is nondecreasing, we can compute that
\begin{align*}
    \left|\mu-\mu^\j\right|_{L^1} = \sum_{k=1}^{\Delta^{-1}}\int_{(k-1)\Delta}^{k\Delta}\left|\mu(s) - \Delta^{-1}\int_{(k-1)\Delta}^{k\Delta}\mu(r) \d r\right|\d s \\
    = 2\Delta^{-1} \sum_{k=1}^{\Delta^{-1}}\int_{(k-1)\Delta}^{k\Delta}\int_{(k-1)\Delta}^{s}(\mu(s)-\mu(r))\d r\d s.
\end{align*}
For $B>0$ to be chosen, using again the monotonicity of $\mu$, we have
\begin{align*}
    2\Delta^{-1} \sum_{k=1}^{\Delta^{-1}}\int_{(k-1)\Delta}^{k\Delta}\int_{(k-1)\Delta}^{s}(\mu(s)-\mu(r))\mathds{1}_{\mu(s)>B}\d r\d s\leq 2\sum_{k=1}^{\Delta^{-1}}\int_{(k-1)\Delta}^{k\Delta}\mu(s)\mathds{1}_{\mu(s)>B} \d s\\
    = 2\int_0^1\mu(s)\mathds{1}_{\mu(s)>B} \d s\leq \frac{2}{B}|\mu|^2_\cH.
\end{align*}
For the other half, changing variables and interchanging summations, we have
\begin{align*}
    &2\Delta^{-1} \sum_{k=1}^{\Delta^{-1}}\int_{(k-1)\Delta}^{k\Delta}\int_{(k-1)\Delta}^{s}(\mu(s)-\mu(r))\mathds{1}_{\mu(s)\leq B}\d r\d s
    \\
    &= 2\Delta^{-1} \int_{0}^{\Delta}\int_{0}^{s}\d r\d s\sum_{k=1}^{\Delta^{-1}}(\mu((k-1)\Delta+s)-\mu((k-1)\Delta+r))\mathds{1}_{\mu((k-1)\Delta+s)\leq B}
\end{align*}
Setting $\kappa(s) = \max\{k:\mu((k-1)\Delta+s)\leq B\}$. Then, the summation above is equal to
\begin{align*}
    \sum_{k=1}^{\kappa(s)}(\mu((k-1)\Delta+s)-\mu((k-1)\Delta+r))\leq \mu((\kappa(s)-1)\Delta+s) \leq B
\end{align*}
where in the first inequality, we used the fact that $-\mu((k-1)\Delta+r) + \mu((k-2)\Delta + s)\leq 0$ due to $s-r\leq \Delta$, and that $\mu\geq 0$.
Hence, we obtain
\begin{align*}
    2\Delta^{-1} \sum_{k=1}^{\Delta^{-1}}\int_{(k-1)\Delta}^{k\Delta}\int_{(k-1)\Delta}^{s}(\mu(s)-\mu(r))\mathds{1}_{\mu(s)\leq B}\d r\d s \leq 2\Delta^{-1} \int_{0}^{\Delta}\int_{0}^{s}B\d r\d s = \Delta B.
\end{align*}
In conclusion, we have $|\mu-\mu^\j|_{L^1} \leq \frac{2}{B}|\mu|^2_\cH + \Delta B$. Optimizing over $B$, we get~\eqref{e.mu-mu^j}.
\end{proof}

\small
\bibliographystyle{abbrv}
\newcommand{\noop}[1]{} \def\cprime{$'$}

\end{document}